\documentclass{amsart}
 
\usepackage{amsmath, amsfonts, mathrsfs}
\usepackage{amssymb,latexsym, fontenc}
\usepackage{enumerate}
\usepackage{bbm} 
\usepackage[dvipsnames]{xcolor}
\usepackage[T1]{fontenc}
\usepackage{libertine}
\usepackage{hyperref} 

\usepackage[top=3.1cm, bottom=3.1cm, left=3.1cm, right=3.1cm, includefoot, heightrounded]{geometry}

\newtheorem{theorem}{Theorem}[section]

\newtheorem{lemma}[theorem]{Lemma}

\theoremstyle{definition}
\newtheorem{definition}[theorem]{Definition}

\numberwithin{equation}{section}

\begin{document}

\title[A Fuzzy Radon-Nikod\'{y}m theorem]{A Radon-Nikod\'{y}m theorem 
for fuzzy number multimeasures \\ in locally convex spaces}

\author{ Sokol Kaliaj}

\address{University of Elbasan,
Science Natural Faculty,
Mathematics Department,
Elbasan, Albania}

\email{printimesokol@gmail.com}

\thanks{}



\begin{abstract}
In this paper,  we present a Radon-Nikodym theorem for dominated fuzzy number multimeasures 
defined on $\Sigma$ and taking values in  $\mathcal{F}_{cwk}(X)$, 
where $(\Omega, \Sigma, \mu)$ is a complete finite measure space  
and $\mathcal{F}_{cwk}(X)$ is the family of generalized fuzzy numbers on a Hausdorff locally convex topological vector space $X$, 
i.e., 
a map $u: X \to [0,1]$ is a member of $\mathcal{F}_{cwk}(X)$, if and only if  
for each $\alpha \in (0,1]$ the $\alpha$-cut $[u]^{\alpha}$ of $u$ is a member of the family $cwk(X)$ of all nonempty convex and weakly compact subsets of $X$, ($[u]^{\alpha} = \{x \in X: u(x) \geq \alpha \}$).
\end{abstract}

\maketitle

\section{Introduction}

One of the most interesting problems arising when we deal with fuzzy number multimeasures
is the representation of a fuzzy number multimeasure $M$ as Pettis integral of a fuzzy map $\Gamma$, 
which is said to be 
a fuzzy Radon-Nikodym derivative of $M$. 
The Radon-Nikod\'{y}m theorems for fuzzy number multimeasures are intimate related to  
Radon-Nikod\'{y}m theorems for multimeasures. 
The first Radon-Nikod\'{y}m theorems for multimeasures  
in non-separable locally convex topological vector spaces 
has been obtained by  
B.~Cascales, V.~Kadets and J.~Rodriguez 
\cite[Theorem 3.1\text{ and }Theorem 4.1]{CASC_RNT}. 
The most general forms of the Radon-Nikod\'{y}m theorem in non-separable Banach spaces 
for multimeasures in terms of Pettis integrable multifunctions has been proved by 
K.~Musial \cite{MUS_Conjugate_Banach_Spaces, MUS_RNT_Rich_Multimeasures_I, MUS_RNT_Rich_Multimeasures_II}. 
Some other results can be found in  the papers 
\cite{DIP_RNT, CAND_RNT_GOULD, DIP_Bartle, KAL_ZAA_2024, KAL_JMAAI, KAL_INDIA}.

The first results concerning  
the Radon-Nikod\'{y}m theorems for fuzzy number  multimeasures  
in non-separ-able Banach spaces 
have been proved in \cite[Theorem 3.9\text{ and }Theorem 3.10]{KAL_Fuzzy_SystemI}, 
where  \cite[Theorem 3.10]{KAL_Fuzzy_SystemI} is a Radon-Nikodym theorem for dominated fuzzy number multimeasures in Banach spaces. 
The most general form of the Radon-Nikod\'{y}m theorem 
for fuzzy number multimeasures has been proved in 
\cite[Theorem 5.4]{KAL_Fuzzy_SystemII}. 
A Radon-Nikodym theorem for fuzzy number multimeasure in a non-separable Frechet space has been proved in  
\cite[Theorem 4.6]{KAL_FNM_Frechet}. 
Some other results 
devoted to this topic
can be found in  the papers 
\cite{BAN, STOJ_RNT, WU_RNT, ZHOU, XIA_Measure}.

In this paper, we first prove that if 
and $M: \Sigma \to cwk(X)$ is a dominated multimeasure, i.e., there exists $Q \in cwk(X)$ such that 
\begin{equation*}
\begin{split}
M(A) \subset \mu(A). Q
~\text{ for every }A \in \Sigma,
\end{split}
\end{equation*}
then there exists a multifunction $\Gamma : \Omega \to cwk(X)$ Pettis integrable with 
$M(A) = (P)\int_{A} \Gamma(t) d\mu$ for every $A \in \Sigma$, see Lemma \ref{L_Radon_Nikodym_Dominant}. 
Applying Lemma \ref{L_Radon_Nikodym_Dominant} we will prove that if  
$M : \Sigma \to \mathcal{F}_{cwk}(X)$ is a dominated fuzzy number multimeasure, 
then there exists 
a fuzzy map $\Gamma : \Omega \to \mathcal{F}_{cwk}(X)$ Pettis integrable 
with $M(A) = (FP)\int_{A} \Gamma(t) d \mu$ for all $A \in \Sigma$, 
see Theorem \ref{T_Fuzzy_RNT_Dominant}.

\section{The basic concepts and facts}

Throughout this paper, 
$X$ is a non-separable complete Hausdorff locally convex topological vector space with topological dual $X^{*}$, 
$(\Omega, \Sigma, \mu)$ is a complete finite measure space. 
We denote by $2^{X}$ the family of all non-empty subsets of $X$ 
and by 
$bcc(X)$ ($cwk(X)$) is denoted the subfamily of $2^{X}$ of
all bounded convex closed (convex weakly compact)  subsets of $X$. 
Given $C \in 2^{X}$ we define the \textit{support} of $C$ by 
\begin{equation*}
\begin{split}
\sigma(x^{*}, C) = \sup \{ \langle x^{*}, x \rangle : x \in C \}
~\text{ for each }x^{*} \in X^{*}.
\end{split}
\end{equation*}

\begin{definition}\label{D_Multimeasures}
A map $M: \Sigma \to cwk(X)$ is called a \textit{multimeasure},  
if the map $\sigma(x^{*}, M(\cdot))$ is a countable additive real valued measure for every 
$x^{*} \in X^{*}$. 
\end{definition}

\begin{definition}\label{D_Pettis_Multi}
A multifunction $\Gamma: \Omega \to bcc(X)$ is said to be
\textit{Pettis integrable}, 
if 
\begin{itemize}
\item
$\sigma(x^{*}, \Gamma(\cdot))$ is Lebesgue integrable for every $x^{*} \in X^{*}$,
\item
for each $A \in \Sigma$ there exists $M(A) \in cwk(X)$ such that
\begin{equation*}
\begin{split}
\sigma
\left ( 
x^{*}, M(A) 
\right )
=
\int_{A} \sigma(x^{*}, \Gamma(t)) d \mu
\quad
\text{for every }x^{*} \in X^{*}.
\end{split}
\end{equation*} 
\end{itemize}
We call $M(A)$ the \textit{Pettis integral} of $\Gamma$ over $A$ and 
set $(P)\int_{A} \Gamma(t) d \mu = M(A)$. 
It is clear that the map 
$M : \Sigma \to cwk(X))$ defined by $M(A) = (P)\int_{A} \Gamma(t) d \mu$ for all $A \in \Sigma$, 
is a multimeasure. 
\end{definition}
A presentation of the general theory of Pettis integration of vector functions can be
found in \cite{TAL}, \cite{MUS_Pettis_Integral_Functions} and \cite{MUS_Elsevier_Functions}.
The  theory of Pettis integrable multifunctions goes back to the monograph by Castaing and Valadier \cite{CAST}.
For the general theory of Pettis integrable multifunctions in non-separable Banach spaces 
we use the papers 
\cite{EL}, \cite{CASC_SelectorsA, CASC_SelectorsB} and 
\cite{MUS_Pettis_MultiF, MUS_Approximation}.

Let $\rho :\Sigma \to \Sigma$ be a lifting on 
$(\Omega, \Sigma, \mu)$, 
see \cite[Theorem 3, p.46]{ION}.  
Note that $\rho(\Sigma)$ is a sub-algebra of $\Sigma$ such that  
$\mu(A) >0$ whenever $A \in  \rho(\Sigma) \setminus \{ \emptyset\}$. 
For any $t \in \Omega$, we write 
$\Sigma(t) = \{A \in \rho(\Sigma) : t \in A\}$, and for each $A \in \Sigma(t)$, 
\begin{equation*}
\begin{split}
\Sigma(t,A) = \{B \in \Sigma(t) : B \subset A\} 
\end{split}
\end{equation*} 
and define the partial ordering $\preceq_{t}$ on $\Sigma(t)$ by 
saying that $A \preceq_{t} B$ if and only if $A \supset B$. 
Since 
\begin{equation*}
\begin{split}
A, B \in \Sigma(t) \Rightarrow A \cap B  \in \Sigma(t), 
\end{split}
\end{equation*}
it follows that $\Sigma(t)$ is a directed set with respect 
to $\preceq_{t}$ (cf. \cite[p.65]{KELL}).


\begin{definition}\label{D_Lvarphi}
Let $\nu: \Sigma \to X$ be a mapping and 
let $t \in \Omega$. 
For each $A \in \Sigma(t)$, we write 
$\Delta \nu(t,A) = \frac{\nu(A)}{\mu(A)}$,  
and then, 
\begin{equation*} 
\begin{split}
L_{\nu}(t) 
=
\bigcap_{A \in \Sigma(t)} 
\overline{L_{\nu}(t, A)}^{\sigma(X,X^{*})},
\end{split}
\end{equation*}
where 
$L_{\nu}(t, A) 
= 
\{\Delta \nu(t,B) : B \in \Sigma(t,A)\}
$ 
and 
$\overline{L_{\nu}(t, A)}^{\sigma(X,X^{*})}$ 
is the closure of 
$L_{\nu}(t, A)$ with respect to the weak topology in $X$. 
By \cite[Theorem 7, p.72]{KELL} follows that $L_{\nu}(t)$ is the set 
of all weak limit points of the net 
$(\Delta \nu(t,A))_{A \in \Sigma(t)}$. 
We say that $L_{\nu}(t)$ is the \textit{density} of $\nu$ at the point $t \in \Omega$. 
\end{definition}

The following lemma follows immediately from \cite[Lemma 3.7]{CASC_RNT}.

\begin{lemma}\label{L_Kupka} 
If a map $f: \Omega \to \mathbb{R}$ 
is Lebesgue integrable 
with $\nu(A) = \int_{A} f(t) d \mu$ for all $A \in \Sigma$, 
then there exists $Z \in \Sigma$ with $\mu(Z)=0$ such that 
the net $(\Delta \nu (t,A))_{A \in \Sigma(t)}$ 
converges to $f(t)$ for every 
$t \in \Omega \setminus Z$.
\end{lemma}

\begin{definition}\label{D_LMultifunction}
Let $M: \Sigma \to 2^{X}$ be a mapping 
and let $t \in \Omega$. 
For each $A \in \Sigma(t)$ we write 
$\Delta M(t,A) = \frac{M(A)}{\mu(A)}$,  
and then, 
\begin{equation*}
\begin{split}
L_{M}(t) =\bigcap_{A \in \Sigma(t)} \overline{L_{M}(t, A)}, 
\end{split}
\end{equation*} 
where 
\begin{equation*}
\begin{split}
L_{M}(t, A) = 
\text{co}
\left ( 
A_{M}(t,A)
\right ),
\quad
A_{M}(t,A)= \bigcup_{B \in \Sigma(t,A)} \Delta M(t,B),
\end{split}
\end{equation*}
and $\overline{L_{M}(t, A)}$ 
is the closure of the convex set $L_{M}(t, A)$ in $X$.  
We say that $L_{M}(t)$ is the \textit{density} of $M$ at the point $t \in \Omega$. 
\end{definition}



\begin{definition}\label{D_Fuzzy_Number}
A mapping $u: X \to [0,1]$ is said to be a \textit{fuzzy set} on $X$ (see Zadeh \cite{ZAD}).
For each $\alpha \in (0,1]$, denote by $[u]^{\alpha} = \{x \in X: u(x) \geq \alpha\}$ 
the $\alpha$-level set of $u$. 
We say that the fuzzy set $u$ on $X$ is a \textit{generalized fuzzy number} on $X$, 
if $[u]^{\alpha} \in bcc(X)$ for all $\alpha \in (0,1]$.  
We denote by $\mathcal{F}_{bcc}(X)$ the set of all generalized fuzzy numbers 
on $X$. 
In particular, 
by $\mathcal{F}_{cwk}(X)$ is denoted the set of all generalized fuzzy numbers 
$u \in \mathcal{F}_{bcc}(X)$ such that 
$[u]^{\alpha} \in cwk(X)$  for all $\alpha \in (0,1]$. 
We write $\widetilde{\theta} = \mathbbm{1}_{\{\theta\}}$, 
where  $\mathbbm{1}_{\{\theta\}}$ is the characteristic function of the set $\{\theta\}$ 
and $\theta$ is the zero vector in $X$. 
\end{definition}

It is easy to see that \cite[Lemma 3.1]{KAL_Fuzzy_SystemI} 
holds also for the complete Hausdorff locally convex space $X$. 

\begin{lemma}\label{L_REP_Fuzzy_Numbers}
If $u \in \mathcal{F}_{bcc}(X)$ $(u \in \mathcal{F}_{cwk}(X))$, then the following are true:
\begin{itemize}
\item[(i)]
$[u]^{\alpha} \in bcc(X)$ $([u]^{\alpha} \in cwk(X))$ for all $\alpha \in (0,1]$, 
\item[(ii)]
$[u]^{\beta} \subset [u]^{\alpha}$ whenever $0 < \alpha \leq \beta \leq 1$, 
\item[(iii)]
if $(\alpha_{k})$ is a nondecreasing sequence with terms in $(0,1]$ converging to $\alpha \in (0,1]$, then
\begin{equation*}
\begin{split}
[u]^{\alpha} = \bigcap_{k=1}^{+\infty} [u]^{\alpha_{k}}.
\end{split}
\end{equation*}
\end{itemize}
Conversely, 
if $\{v_{\alpha}: \alpha \in (0, 1]\}$ 
is a family satisfying $(i)$,$(ii)$ and $(iii)$,
then there is the unique generalized fuzzy number $v \in \mathcal{F}_{bcc}(X)$ $(v \in \mathcal{F}_{cwk}(X))$ such that 
$[v]^{\alpha} = v_{\alpha}$ for every $\alpha \in (0, 1]$.
\end{lemma}

Hence, the following auxiliary lemma can be proved in much the same way as \cite[Lemma 3.2]{KAL_Fuzzy_SystemI}.

\begin{lemma}\label{L_Fuzzy_Numbers_CWK}
Let $X$ be a complete Hausdorff locally convex space 
and let $\{u_{\alpha} \in 2^{X}: \alpha \in (0, 1]\}$ be 
a family.  
Assume that
\begin{itemize}
\item[(i)]
$u_{\alpha} \in cwk(X)$ for all $\alpha \in (0,1]$, 
\item[(ii)]
$u_{\beta} \subset u_{\alpha}$ whenever $0 < \alpha \leq \beta \leq 1$.
\end{itemize}
Then there exists a unique generalized fuzzy number $v \in \mathcal{F}_{cwk}(X)$ such that 
for each $\alpha \in (0, 1]$, we have 
\begin{equation*}
\begin{split}
[v]^{\alpha} = \bigcap_{n=1}^{+\infty} u_{q_{n}^{\alpha}},
\end{split}
\end{equation*}
where $(q_{n}^{\alpha})$ determines $\alpha$ (see \cite[Definition 2.1]{KAL_Fuzzy_SystemI}). 
\end{lemma}


\begin{definition}\label{D_Fuzzy_Number_Measure} 
A fuzzy map 
$M : \Sigma \to \mathcal{F}_{cwk}(X)$ 
is called a \textit{fuzzy number multimeasure}, 
if for each $\alpha \in (0,1]$ the map 
\begin{equation*}
\begin{split}
M^{\alpha} : \Sigma \to cwk(X), 
\quad
M^{\alpha}(A) = [M(A)]^{\alpha} = \{x \in X : M(A)(x) \geq \alpha\}
\end{split}
\end{equation*}
is a multimeasure. 
\end{definition}


\begin{definition}\label{D_Pettis_Fuzzy}
Given a fuzzy map  $\Gamma : \Omega \to \mathcal{F}_{bcc}(X)$ 
we can define
\begin{equation*}
\begin{split}
\Gamma^{\alpha} : \Omega \to cwk(X), 
\quad
\Gamma^{\alpha}(t) = [\Gamma(t)]^{\alpha} = \{x \in X : \Gamma(t)(x) \geq \alpha\}
\quad
(\alpha \in (0,1]). 
\end{split}
\end{equation*}
We say that the fuzzy map 
$\Gamma : \Omega \to \mathcal{F}_{bcc}(X)$ 
is 
\textit{Pettis integrable} if 
\begin{itemize}
\item
the map $\sigma(x^{*}, \Gamma^{\alpha}(\cdot))$ is Lebesgue integrable for each $x^{*} \in X^{*}$ and $\alpha \in (0,1]$,
\item
for each $A \in \Sigma$ there exists 
a generalized fuzzy number $M(A) \in \mathcal{F}_{cwk}(X)$ such that  
\begin{equation*}
\begin{split}
\sigma \left ( x^{*}, [M(A)]^{\alpha} \right ) 
= 
\int_{A} \sigma(x^{*}, \Gamma^{\alpha}(t)) d \mu
~\text{ for all }x^{*} \in X^{*}, \alpha \in (0,1].
\end{split}
\end{equation*}
\end{itemize}
We call $M(A)$ the \textit{fuzzy Pettis integral} of $\Gamma$ over $A$ and 
set $(FP)\int_{A} \Gamma(t) d \mu = M(A)$. 
It is clear that the map 
\begin{equation*}
\begin{split}
M : \Sigma \to \mathcal{F}_{cwk}(X), 
\quad
M(A) = (FP)\int_{A} \Gamma(t) d \mu
\end{split}
\end{equation*}
is a fuzzy number multimeasure.
\end{definition}

\section{The main result}

The main result is Theorem \ref{T_Fuzzy_RNT_Dominant}. 
Let us start with two auxiliary lemmas.

\begin{lemma}\label{L_Intersection_LSC}
Let $\mathcal{I}$ be an index set directed under a (reflexive, transitive,
anti-symmetric) relation "$\leq$", let $\{C_{i} \in cwk(X): i \in \mathcal{I}\}$ be a family such that 
$C_{i} \supset C_{i'}$ whenever $i \leq i'$. 
\begin{itemize}
\item[(i)]
If $C = \bigcap_{i \in \mathcal{I}} C_{i}$, then 
\begin{equation}\label{eqL_Intersection_LSC.0}
\begin{split}
\lim_{i \in \mathcal{I}} \sigma(x^{*}, C_{i}) = \sigma(x^{*},C)
~\text{ for every }x^{*} \in X^{*}.
\end{split}
\end{equation}
\item[(ii)]
If \eqref{eqL_Intersection_LSC.0} holds for every $x^{*} \in X^{*}$, 
then $C = \bigcap_{i \in \mathcal{I}} C_{i}$.
\end{itemize}
\end{lemma}
\begin{proof}
$(i)$ 
We first fix an index $i_{0} \in \mathcal{I}$. 
Then,
\begin{equation}\label{eqL_Intersection_LSC.1}
\begin{split}
C = \bigcap_{i \in \mathcal{I}} C_{i} = \bigcap_{i \in \mathcal{I}_{0}} C_{i}
\qquad
(\mathcal{I}_{0} = \{i \in \mathcal{I}: i \geq i_{0}\}),
\end{split}
\end{equation}
and by \cite[Theorem 1, p.136]{KELL} follows that $C \neq \emptyset$.

According to \cite[Theorem 7, p.140]{KELL}, $C_{i}$ is $\sigma(X,X^{*})$-closed 
for every $i \in \mathcal{I}_{0}$. 
Hence, the convex set $C$ is $\sigma(X,X^{*})$-closed, 
and since each net with terms in $C \subset C_{i_{0}}$ 
has a subnet which converges to a 
point which belongs to $C$, 
it follows that $C \in cwk(X)$ (cf. \cite[Theorem 2, p.136]{KELL}).

By \eqref{eqL_Intersection_LSC.1}, we have
\begin{equation*}
\begin{split}
\sigma(x^{*}, C) \leq \inf_{i \in \mathcal{I}_{0}} \sigma(x^{*}, C_{i}) 
= \lim_{i } \sigma(x^{*}, C_{i})
~\text{ for every }x^{*} \in X^{*}. 
\end{split}
\end{equation*}
Fix an arbitrary $x^{*} \in X^{*}$ and suppose that there exists $r \in \mathbb{R}$ such that 
\begin{equation*}
\begin{split}
\sigma(x^{*}, C) < r < \inf_{i \in \mathcal{I}_{0}} \sigma(x^{*}, C_{i}). 
\end{split}
\end{equation*}
We can choose $\varepsilon >0$ so that 
\begin{equation*}
\begin{split}
\sigma(x^{*}, C) < r- \varepsilon <r<r +\varepsilon < \inf_{i \in \mathcal{I}_{0}} \sigma(x^{*}, C_{i}). 
\end{split}
\end{equation*}
Then, for each $i \in \mathcal{I}_{0}$ there is $x_{i} \in C_{i}$ such that 
\begin{equation}\label{eqL_Intersection_LSC.3}
\begin{split}
\langle x^{*}, x_{i} \rangle > r+ \varepsilon.
\end{split}
\end{equation}
Since $(x_{i})_{i \in \mathcal{I}_{0}}$ is a  net with terms in weakly compact set $C_{i_{0}} \in cwk(X)$, 
it follows by \cite[Theorem 2, p. 136]{KELL} that 
there is a subnet $(y_{j})_{j \in \mathcal{J}}$ of the net $(x_{i})_{i \in \mathcal{I}_{0}}$ 
converging to a point $x_{0} \in C_{i_{0}}$ 
with respect to the $\sigma(X,X^{*})$-topology. 
From the fact that $(y_{j})_{j \in \mathcal{J}}$ is a subnet of the net $(x_{i})_{i \in \mathcal{I}_{0}}$, 
it follows that there exists a function $u: \mathcal{I}_{0} \to \mathcal{J}$ such that 
\begin{itemize}
\item
$y_{j} = x_{u(j)}$ whenever $j \in \mathcal{J}$, 
\item
for each $i \in \mathcal{I}$ there exists $j_{i} \in \mathcal{J}$ such that 
$j \succeq j_{i} \Rightarrow u(j) \geq i$.
\end{itemize}

Note that for each $i \in \mathcal{I}_{0}$, 
$C_{i} = \overline{C_{i}} = \overline{C_{i}}^{\sigma(X,X^{*})}$ 
(cf. 
\cite[Proposition 8, p.34]{ROB} 
or 
\cite[Corollary 2, p.65]{SCH}), 
and  
\begin{equation*}
\begin{split}
\{x_{i'}: i' \geq i\} \subset C_{i} 
\Rightarrow 
\{y_{j} : j \succeq j_{i}\} \subset C_{i},
\end{split}
\end{equation*}
and since the net $(y_{j})_{j \in \mathcal{J}}$ converges weakly to $x_{0}$, 
it follows that $x_{0} \in C = \bigcap_{i \in \mathcal{I}_{0}} C_{i}$, 
whence
\begin{equation*}
\begin{split}
\langle x^{*}, x_{0} \rangle \leq \sigma(x^{*}, C) < r - \varepsilon.
\end{split}
\end{equation*}
On the other hand, 
\begin{equation*}
\begin{split}
\lim_{j } \langle x^{*}, y_{j} \rangle = \langle x^{*}, x_{0} \rangle \geq r +  \varepsilon
\quad
(\text{see \eqref{eqL_Intersection_LSC.3}}).
\end{split}
\end{equation*}
This contradiction shows that \eqref{eqL_Intersection_LSC.0} holds.

$(ii)$ 
Assume that \eqref{eqL_Intersection_LSC.0} holds and let $D = \bigcap_{i \in \mathcal{I}} C_{i}$. 
Then, we obtain by $(i)$ that 
\begin{equation*}
\begin{split}
\sigma(x^{*}, D) = \lim_{i \in \mathcal{I}} \sigma(x^{*}, C_{i}) = \sigma(x^{*}, C)
~\text{ for every }x^{*} \in X^{*}. 
\end{split}
\end{equation*}
Hence, 
by the Hahn-Banach separation
theorem (\cite[Theorem V.2.10]{DUN}) follows that $D=C$, 
and this completes the proof.
\end{proof}

\begin{lemma}\label{L_Radon_Nikodym_Dominant}
Let $M: \Sigma \to cwk(X)$ be a  multimeasure. 
Assume that there exists $Q \in cwk(X)$ such that 
$M(A) \subset \mu(A) Q$ for all $A \in \Sigma$. 
Then,
\begin{itemize}
\item[(i)]
$L_{M}(t) \subset Q$ and $L_{M}(t) \neq \emptyset$ for all $t \in \Omega$,
\item[(ii)]
the set valued function  
$\Gamma: \Omega \to cwk(X)$ defined by $\Gamma(t) = L_{M}(t)$ for all $t \in \Omega$ 
is Pettis integrable with 
$M(A) = (P)\int_{A} \Gamma(t) d \mu$ for every $A \in \Sigma$.
\end{itemize}
\end{lemma}
\begin{proof}
Since $\overline{L_{M}(t, A)} = \overline{L_{M}(t, A)}^{\sigma(X,X^{*})} \in cwk(X)$ 
for all $A \in \Sigma(t)$ ($t \in \Omega$), 
it follows that $L_{M}(t) \in cwk(X)$ for all $t \in \Omega$.
Hence, we obtain by Lemma \ref{L_Intersection_LSC} that
\begin{equation*}
\begin{split}
\sigma(x^{*}, L_{M}(t)) = \lim_{A \in \Sigma(t)} \sigma(x^{*}, L_{M}(t,A)) 
\qquad
(t \in \Omega, x^{*} \in X^{*}, A \in \Sigma), 
\end{split}
\end{equation*}
and since 
$\sigma(x^{*}, L_{M}(t,A)) = \sigma(x^{*}, A_{M}(t,A))$ 
for every $t \in \Omega, x^{*} \in X^{*}$ and $A \in \Sigma$, 
it follows that
\begin{equation}\label{eqL_Radon_Nikodym_Dominant.1}
\begin{split}
\sigma(x^{*}, L_{M}(t)) = \lim_{A \in \Sigma(t)} \sigma(x^{*}, A_{M}(t,A))
\qquad
(t \in \Omega, x^{*} \in X^{*}, A \in \Sigma).
\end{split}
\end{equation}

Fix an arbitrary $x^{*} \in X^{*}$. 
Then, the map $\sigma(x^{*}, M(\cdot))$ is a $\mu$-continuous real valued measure. 
Hence, there exists a Lebesgue integrable function $h: \Omega \to \mathbb{R}$ which is 
a Radon-Nikdym derivative of the measure $\sigma(x^{*}, M(\cdot))$, i.e., 
\begin{equation}\label{eqL_Radon_Nikodym_Dominant.2} 
\begin{split}
\sigma(x^{*}, M(A)) = \int_{A} h(t) d \mu 
~\text{ for every }A \in \Sigma.
\end{split}
\end{equation}
By Lemma \ref{L_Kupka} 
there exists $Z \in \Sigma$ with $\mu(Z)=0$ such that
\begin{equation}\label{eqL_Radon_Nikodym_Dominant.3}
\begin{split}
\lim_{A \in \Sigma(t)} \sigma(x^{*}, \Delta M(t, A)) =  h(t)
~\text{ for all }t \in \Omega \setminus Z.
\end{split}
\end{equation} 

Since
\begin{equation*}
\begin{split}
x^{*}(A_{M}(t,A)) 
= \bigcup_{B \in \Sigma(t, A)} x^{*}(\Delta M(t, B))
\qquad
(t \in \Omega,  A \in \Sigma).
\end{split}
\end{equation*}
we obtain 
\begin{equation*}
\begin{split}
\inf_{A \in \Sigma(t)} \sigma(x^{*}, A_{M}(t,A)) 
= \inf_{A \in \Sigma(t)} \left ( \sup_{B \in \Sigma(t, A)} \sigma(x^{*}, \Delta M(t, B)) \right )
\qquad
(t \in \Omega),
\end{split}
\end{equation*}
and since
\begin{equation*}
\begin{split}
\inf_{A \in \Sigma(t)} \sigma(x^{*}, A_{M}(t,A))
=
\lim_{A \in \Sigma(t)} \sigma(x^{*}, A_{M}(t,A))
\qquad
(t \in \Omega),
\end{split}
\end{equation*}
and
\begin{equation*}
\begin{split}
\inf_{A \in \Sigma(t)} \left ( \sup_{B \in \Sigma(t, A)} \sigma(x^{*}, \Delta M(t, B)) \right )
= \lim_{A \in \Sigma(t)} \sigma(x^{*}, \Delta M(t, A))
\qquad
(t \in \Omega \setminus Z),
\end{split}
\end{equation*}
it follows by \eqref{eqL_Radon_Nikodym_Dominant.1} and \eqref{eqL_Radon_Nikodym_Dominant.3} that 
\begin{equation*}
\begin{split}
\sigma(x^{*}, L_{M}(t)) = \lim_{A \in \Sigma(t)} \sigma(x^{*}, A_{M}(t,A)) 
= 
h(t)
~\text{ for all }t \in \Omega \setminus Z.
\end{split}
\end{equation*}
Thus, $\sigma(x^{*}, L_{M}(t)) = h(t)$ at almost all $t \in \Omega$. 
Hence, the map  $\sigma(x^{*}, L_{M}(\cdot))$ is Lebesgue integrable 
and by \eqref{eqL_Radon_Nikodym_Dominant.2} follows that
\begin{equation*}
\begin{split}
\sigma(x^{*}, M(A)) = \int_{A} h(t) d \mu = \int_{A} \sigma(x^{*}, L_{M}(t)) d \mu
~\text{ for every }A \in \Sigma.
\end{split}
\end{equation*}
Since $x^{*}$ was arbitrary, the last result means that 
the multifunction $\Gamma(\cdot) = L_{M}(\cdot)$ is Pettis integrable 
with $M(A) = (P)\int_{A} \Gamma(t) d \mu$ for all $A \in \Sigma$, 
and this completes the proof.
\end{proof}

We are now ready to present the main result.

\begin{theorem}\label{T_Fuzzy_RNT_Dominant} 
Let $M: \Sigma \to \mathcal{F}_{cwk}(X)$ be 
a fuzzy number multimeasure. 
Assume that there is a generalized fuzzy number $Q \in \mathcal{F}_{cwk}(X)$ such that 
\begin{equation*}
\begin{split}
M^{\alpha}(A) \subset \mu(A) [Q]^{\alpha} 
~\text{ for all }A \in \Sigma,~\alpha \in (0,1].
\end{split}
\end{equation*}
Then, there exists a fuzzy map $\Gamma : \Omega \to \mathcal{F}_{cwk}(X)$ 
Pettis integrable  with 
$M(A) = (FP)\int_{A} \Gamma(t) d\mu$ for all $A \in \Sigma$.
\end{theorem}
\begin{proof}
According to Lemma \ref{L_Radon_Nikodym_Dominant}, 
for each $\alpha \in (0,1]$ the map $\Gamma_{\alpha}(\cdot) = L_{M^{\alpha}}(\cdot)$ 
is Pettis integrable with 
\begin{equation}\label{eqT_Fuzzy_RNT_Dominant.1}
\begin{split}
\sigma(x^{*}, [M(A)]^{\alpha}) 
= 
\int_{A} \sigma(x^{*}, \Gamma_{\alpha}(t)) d \mu
\qquad
(x^{*} \in X^{*}, A \in \Sigma, \alpha \in (0,1]).  
\end{split}
\end{equation}

We are going to prove that for each $t \in \Omega$ the family 
$\{\Gamma_{\alpha}(t): \alpha \in (0,1]\}$ 
satisfies $(i)$ and $(ii)$ in  Lemma \ref{L_Fuzzy_Numbers_CWK}
\begin{itemize}
\item
Since $\Gamma_{\alpha}(t) \in cwk(X)$ for each $t \in \Omega$ and $\alpha \in (0,1]$, 
it follows that $(i)$ in Lemma \ref{L_Fuzzy_Numbers_CWK} is satisfied.
\item
Let $\alpha, \beta \in (0,1]$ be such that 
$0 < \alpha \leq \beta \leq 1$ and let $t \in \Omega$. 
Then, 
$\Gamma_{\beta}(t) = L_{M^{\beta}}(t) \subset L_{M^{\alpha}}(t) = \Gamma_{\alpha}(t)$.   
This means that 
the family 
$\{\Gamma_{\alpha}(t): \alpha \in (0,1]\}$ 
satisfies also $(ii)$ in Lemma \ref{L_Fuzzy_Numbers_CWK} whenever $t \in \Omega$. 
\end{itemize}

Since $\{\Gamma_{\alpha}(t): \alpha \in (0,1]\}$  
satisfies $(i)$ and $(ii)$ in Lemma \ref{L_Fuzzy_Numbers_CWK}
for all $t \in \Omega$, 
it follows that there exists a unique 
$\Gamma(t) \in \mathcal{F}_{cwk}(X)$ such that 
\begin{equation*}
\begin{split} 
[\Gamma(t)]^{\alpha}
=
\bigcap_{n=1}^{+\infty} \Gamma_{q_{n}^{\alpha}}(t)
~\text{ for every }\alpha \in (0,1], t \in \Omega,
\end{split}
\end{equation*}
and hence by $(i)$ in Lemma \ref{L_Intersection_LSC} we obtain
\begin{equation*}
\begin{split}
\lim_{n \to \infty} \sigma(x^{*}, \Gamma_{q_{n}^{\alpha}}(t)) = \sigma(x^{*}, [\Gamma(t)]^{\alpha})
\qquad
(\alpha \in (0,1], x^{*} \in X^{*}, t \in \Omega).
\end{split}
\end{equation*}
Consequently, by Dominated Convergence Theorem the map 
$\sigma(x^{*}, [\Gamma(\cdot)]^{\alpha})$ is Lebesgue integrable 
and
\begin{equation*} 
\begin{split}
\int_{A} \sigma(x^{*}, [\Gamma(t)]^{\alpha}) d \mu 
=& 
\lim_{n \to \infty} \int_{A}  \sigma(x^{*}, \Gamma_{q_{n}^{\alpha}}(t)) d \mu 
=
\lim_{n \to \infty} \sigma(x^{*}, M^{q_{n}^{\alpha}}(A))
\\
=&
\sigma(x^{*}, M^{\alpha}(A))
\qquad
(x^{*} \in X^{*}, A \in \Sigma, \alpha \in (0,1]).  
\end{split}
\end{equation*}
This means that the fuzzy map $\Gamma$ is Pettis integrable 
with $M(A) = (FP)\int_{A} \Gamma(t) d \mu$ for all $A \in \Sigma$. 
Thus, $\Gamma$ is a fuzzy Radon-Nikodym derivative of the fuzzy number multimeasure $M$ and the proof is complete.
\end{proof}


\begin{thebibliography}{plain}








\bibitem{BAN} J.~Ban, 
\textit{Radon-Nikod\'{y}m theorem and conditional expectation 
of fuzzy valued measures and variables}, 
Fuzzy Sets and Systems \textbf{34} (1990), 383-392.







\bibitem{CAND_RNT_GOULD}
D. Candeloro, A. Croitoru, A. Gavrilut, A.R. Sambucini, A multi- valued version 
of the Radon-Nikod\'{y}m theorem, via the single-valued Gould integral, 
Australian Journal of Mathematical Analysis and Applications 
\textbf{15/2}, Art. 9, (2018), 1-16.








\bibitem{CAST} C. ~Castaing, ~M. ~Valadier, 
\textit{Convex Analysis and Measurable Multifunctions} ,
Lect. Notes Math., \textbf{580}, Springer-Verlag, Berlin (1977).







\bibitem{CASC_SelectorsA} B.~Cascales, V.~Kadets, and J.~Rodr\'{i}guez,  
\textit{Measurable selectors and set-valued Pettis integral in non-separable Banach spaces}, 
J. Funct. Anal. \textbf{256/3} (2009), 673-699.






\bibitem{CASC_SelectorsB} B.~Cascales, V.~Kadets, and J.~Rodr\'{i}guez, 
\textit{Measurability and selections of multifunctions in Banach spaces}, 
J. Convex Analysis \textbf{17} (2010), 229-240. 






\bibitem{CASC_RNT} B.~Cascales, V.~Kadets, J.~Rodriguez,
\textit{Radon-Nikod\'{y}m theorems for multimeasures in non-separable spaces}, 
Zh. Mat. Fiz. Anal. Geom. \textbf{9/1} (2013), 7-24.






\bibitem{DIP_RNT} L.~Di Piazza, G.~Porcello, 
\textit{Radon-Nikod\'{y}m theorems for
finitely additive multimeasures}, 
Journal of Analysis and its Applications
\textbf{34} (2015), 373-389. 






\bibitem{DIP_Bartle} L. Di Piazza, K. Musial, A.R. Sambucini, 
\textit{Representations of multimeasures via multivalued Bartle-Dunford-Schwartz integral}, 
Journal of Convex Analysis \textbf{(29/4)} (2022), 1119-1148. 






\bibitem{DUN} N. Dunford and J. T. Schwartz, 
\textit{Linear operators. Part I: General theory}, 
Interscience, New York, (1958).






\bibitem{EL} K.~El Amri and C.~Hess, 
\textit{On the Pettis integral of closed valued multifunctions}, 
Set-Valued Anal. \textbf{8(4)} (2000), 329-360.






\bibitem{FOLL} G.~B.~Folland, 
\textit{Real analysis: Modern Techniques and their Applications}, 
A Wiley-Interscience Publication (1999). 









\bibitem{KAL_ZAA_2024} S.~B.~Kaliaj, 
\textit{Radon-Nikod\'{y}m theorems in non-separable Banach spaces}, 
Z. Anal. Anwend. \textbf{43} (2024), 125-142.






\bibitem{KAL_JMAAI} S.~B.~Kaliaj, 
\textit{Extension theorems of additive interval multifunctions with applications}, 
J. Math. Anal. Appl. \textbf{540} (2024).






\bibitem{KAL_INDIA} S.~B.~Kaliaj,
\textit{Radon-Nikod\'{y}m theorems for additive interval multimeasures in Banach spaces},  
Indian J. Pure Appl. Math. (2024).  







\bibitem{KAL_Fuzzy_SystemI} S.~B.~Kaliaj, 
\textit{Radon-Nikod\'{y}m theorems for generalized fuzzy number measures}, 
Fuzzy Sets and Systems \textbf{512} (2025).





\bibitem{KAL_Fuzzy_SystemII} S.~B.~Kaliaj, 
\textit{Radon-Nikod\'{y}m theorems for fuzzy number multimeasures}, 
Fuzzy Sets and Systems \textbf{524} (2026). 



\bibitem{KAL_FNM_Frechet} S.~B.~Kaliaj, 
\textit{A Radon-Nikod\'{y}m theorem for multimeasures in Fr\'{e}chet spaces with application}, 
J. Math. Anal. Appl. \textbf{557/2} (2026).





\bibitem{KELL} J.~L.~Kelly,
\textit{General Topology}, 
Springer New York Van Nostrand, (1955).







\bibitem{MUS_Pettis_Integral_Functions} K.~Musial,
\textit{Topics in the theory of Pettis integration},
Rend. Istit. Mat. Univ. Trieste \textbf{23} (1991), 177-262. 






\bibitem{MUS_Elsevier_Functions} K.~Musial, 
\textit{Pettis integral, in: Handbook of Measure Theory I}, 
Elsevier, Amsterdam (2002), 532-586.






\bibitem{MUS_Pettis_MultiF} K.~Musial,
\textit{Pettis integrability of multifunctions with values in arbitrary Banach spaces}, 
J. Convex Analysis \textbf{18} (2011), 769-810.






\bibitem{MUS_Approximation} K.~Musial,
\textit{Approximation of Pettis integrable multifunctions with values in arbitrary Banach spaces}, 
J. Convex Analysis \textbf{20(3)} (2013), 833-870.





\bibitem{MUS_Conjugate_Banach_Spaces} K.~Musial,
\textit{Multimeasures with values in conjugate Banach spaces and the weak
Radon-Nikod\'{y}m property}, 
J. Convex Analysis \textbf{28/3} (2021), 879-902.






\bibitem{MUS_RNT_Rich_Multimeasures_I} K.~Musial, 
\textit{Lifting Approach to Integral Representations of Rich Multimeasures with Values in Banach Spaces $I$},
J.~Convex Anal., \textbf{(31/1)} (2024), 179-193. 






\bibitem{MUS_RNT_Rich_Multimeasures_II} K.~Musial, 
\textit{Lifting Approach to Integral Representations of Rich Multimeasures with Values in Banach Spaces $II$},
J.~Convex Anal., \textbf{(31/3)} (2024), 867-888. 






\bibitem{OUY} Yao Ouyang, Jun Li, 
\textit{A Radon-Nikod\'{y}m theorem for monotone measures}, 
Fuzzy Sets and Systems \textbf{487} (2024). 




\bibitem{ROB} A.~P.~Robertson and W.~Robertson, 
\textit{Topological Vector Spaces}, 
Cambridge University  Press (1964). 






\bibitem{SCH} H.~H.~Schaefer, W.~P.~Wolff, 
\textit{Topological Vector Spaces},
Originally published by Springer-Verlag New York, (1999).







\bibitem{STOJ_RNT} M.~Stojakovi\'{c}, 
\textit{Fuzzy valued measure}, 
Fuzzy Sets and Systems \textbf{65/1} (1994), 95-104.








\bibitem{TAL}  M.~Talagrand, 
\textit{Pettis Integral and Measure Theory}, 
Mem. Amer. Math. Soc. \textbf{51} (1984).








\bibitem{ION} A.~Ionescu Tulcea and C.~Ionescu Tulcea, 
\textit{Topics in the theory of lifting},
Ergebnisse der Mathematik und ihrer Grenzgebiete, 
Band \textbf{48}, Springer-Verlag New York Inc., New York, (1969).







\bibitem{WU_RNT} Jianrong Wu, Xiaoping Xue, Congxin Wu, 
\textit{Radon-Nikod\'{y}m theorem and Vitali-Hahn-Saks theorem on fuzzy number measures in Banach spaces}, 
Fuzzy Sets and Systems \textbf{117/3} (2001), 339-346. 

 




\bibitem{XIA_Measure} Xue Xiaoping,  Ha Minghu,  Wu Congxin, 
\textit{On the extension of the fuzzy number measures in Banach spaces: 
Part I. Representation of the fuzzy number measures},
Fuzzy sets and system \textbf{78} (1996), 347-354.








\bibitem{ZAD} L.~A.~Zadeh, 
\textit{Fuzzy sets}, 
Inform. and Control \textbf{8} (1965), 338-353.







\bibitem{ZHOU} Cai-Li Zhou and Fu-Gui Shi,
\textit{Lebesgue Decomposition Theorem and Weak Radon-Nikod\'{y}m
Theorem for Generalized Fuzzy Number Measures}, 
Journal of Function Spaces Volume 2015, Article ID 576134, 8 pages,
http://dx.doi.org/10.1155/2015/576134.

\end{thebibliography}
\end{document}